\documentclass{article}

\usepackage{amsmath,amssymb,amsthm,mathrsfs,float}
\usepackage{graphicx}

				\newtheorem{theorem}{Theorem}
                \newtheorem{corollary}[theorem]{Corollary}
                \newtheorem{lemma}[theorem]{Lemma}
                \newtheorem{proposition}[theorem]{Proposition}
\theoremstyle{definition} \newtheorem*{definition}{Definition}
\theoremstyle{definition} 
\theoremstyle{definition} \newtheorem{remark}{Remark}
\theoremstyle{definition}

\begin{document}

\title{Local Coupling Property for Markov Processes with Applications to L\'evy Processes}
\author{Kasra Alishahi\\
Sharif University of Technology\\
Erfan Salavati\\
Faculty of Mathematics and Computer Science,\\
Amirkabir University of Technology (Tehran Polytechnic),\\
P.O. Box 15875-4413, Tehran, Iran.}
\date{}

\maketitle


\begin{abstract}
In this article, we define the new concept of local coupling property for markov processes and study its relationship with distributional properties of the transition probability. In the special case of L\`evy processes we show that this property is equivalent to the absolute continuity of the transition probability and also provide a sufficient condition for it in terms of the L\'evy measure. Our result is stronger than existing results for absolute continuity of L\'evy distributions.
\end{abstract}

\section{Introduction}

The coupling method is a very powerful tool for studying properties of stochastic processes. In case of Markov processes this method is used for proof of convergence to stationary measure. This method can also be used to prove the distributional properties of the transition probability measure of the Markov process.

A coupling between two stochastic objects, means a joint distribution whose marginal distributions are that of those objects. For two stochastic processes, those couplings are concerned in which the two processes coincide eventually. Such couplings are called successful couplings.

\section{Coupling property and its equivalent statements}
Let $\mathbb{R}^+=[0,\infty)$  and  $\{X_t\}_{t\in\mathbb{R}^+}$ be a continuous time markov process on a metric space $(E,\mathcal{E})$ and  $P^t(x,dy)$ be its transition probability measure.

$\mathcal{B}(E)$ and $\mathcal{C}(E)$ denote the space of respectively bounded and continuous real-valued functions on $E$.

For each $f\in\mathcal{B}(E)$ we define
\[ P^t f(x) = \int_E f(y) P^t(x,dy) \]

A function $u(t,x)$ is called space-time harmonic if for all $t,s>0$,
\[ u(s,x) = P^t u(s+t,.) (x) =  \int_E u(s+t,y) P^t(x,dy) \]

For any $x\in E$, Let $\mathbb{P}_x(X\in .)$ be the probability measure on $(E^{\mathbb{R}^+},\mathcal{E}^{\mathbb{R}^+})$ which is induced by the Markov process starting at $x$.

For any probability measure $\mu$ on $E$, let $\mathbb{P}_\mu(.)$ be the induced probability measure by the process starting with initial distribution $\mu$. In other words
\[ \mathbb{P}_\mu(.) = \int \mathbb{P}_x(.) d\mu(x) \]

For any $\ge 0$, the shift operator $\theta_t$ on $(E^{\mathbb{R}^+},\mathcal{E}^{\mathbb{R}^+})$ is defined as
\[ \theta_t X (.) = X(t+.) \]

\begin{definition}[Local Coupling Property]
The Markov process $X_t$ is said to have the local coupling property if for any $x\in E$ and $\epsilon>0$, there exists $\delta>0$ such that for any $y$ with $d(y,x)<\delta$, there exists a coupling $(X_t,Y_t)$ between $\mathbb{P}_x$ and $\mathbb{P}_y$ with the property that for
\[ T=T_{x,y}=\inf\{t\ge 0: X_t=Y_t\} \]
we have $P(T>\epsilon)<\epsilon$.
\end{definition}

Without loss of generality, we can assume that for $t\ge T$, $X_t=Y_t$ because we can assume that the two processes move together after time $T$.

By $\|\mu\|$ we mean the total variation norm of the signed measure $\mu$. We are now ready to state and prove the main theorem of this section.

\begin{theorem}\label{thm:Markov_main}
For a Markov process $X_t$ on a Polish space $E$, the following statements are equivalent:
	\begin{description}
		\item[(i)]
		$X$ has the local coupling property.
		\item[(ii)]
		For any $x\in E$ and $t>0$,
		\[ \lim_{y\to x} \|P^t(y,.)-P^t(x,.)\| = 0 \]
		\item[(iii)]
		For any $x\in E$ and any $t>0$,
		\[ \lim_{y\to x} \|\mathbb{P}_y (\theta_t X\in \cdot)-\mathbb{P}_x (\theta_t X\in \cdot)\| = 0 \]
	\end{description}
\end{theorem}

\begin{proof}
	\begin{description}
		\item[(i) $\implies$ (ii)]
		Assume $t>0$ is given and let $(X_t,Y_t)$ be a coupling of $\mathbb{P}_x$ and $\mathbb{P}_y$ which satisfies the definition of local coupling property for an $\epsilon<t$. Then we have
		\begin{eqnarray*}
			\|P^t(y,.)-P^t(x,.)\| & \le & \mathbb{P}(X_t\ne Y_t)\\
			&=&\mathbb{P}(T >t) \le \mathbb{P}(T >\epsilon) \le \epsilon
		\end{eqnarray*}
Now by letting $\epsilon\to 0$ the statement follows.

		\item[(ii) $\Leftrightarrow$ (iii)]
		Let $\mu=P^t(x,.)$ and $\nu=P^t(y,.)$. We have
		\[ \mathbb{P}_y (\theta_t X\in .) - \mathbb{P}_x(\theta_t X\in .) = \mathbb{P}_\nu (.) - \mathbb{P}_\mu(.) \]
Now let $\rho=\mu+\nu$ and assume that $g_\mu=\frac{d\mu}{d\rho}$ and $g_\nu=\frac{d\nu}{d\rho}$ are respectively the Radon-Nikodym derivatives of  $\mu$ and $\nu$ with respect to $\rho$. Hence we have,
	\[ \mathbb{P}_\nu (.) - \mathbb{P}_\mu(.) = \int \mathbb{P}_x (.) g_\nu(x)d\rho(x) - \int \mathbb{P}_x (.) g_\mu(x) d\rho(x)\]
	\[ \le \int_{g_\nu\ge g_\mu} (g_\nu - g_\mu) d \rho = \frac{1}{2} \|\nu-\mu\| \]
which implies $\|\mathbb{P}_\nu - \mathbb{P}_\mu \| \le \|\nu-\mu \|$. The other direction of the statement is obvious.

		\item[(iii) $\implies$ (i)]
Given $x$ and $\epsilon>0$, there exists $\delta>0$ such that for any $y$ with $d(y,x)<\delta$,
		\[ \|\mathbb{P}_y(\theta_\epsilon X\in .) - \mathbb{P}_x (\theta_\epsilon X \in .) \| <2 \epsilon \]
Now consider the maximal coupling between these two measures and denote it by $(\{X_x(t)\}_{t\ge \epsilon},\{X_y(t)\}_{t\ge \epsilon})$.

Now using the regular conditional probability we extend this coupling to $0\le t <\epsilon$. To do this, we follow the machinery used in \cite{Lindvall}, section 15. Note that by the Polish assumption, there exists a regular version of the conditional probabilities
\[ \mathbb{P}_x(X\in .| \theta_\epsilon X =Z) , \quad \mathbb{P}_y(X\in .| \theta_\epsilon X =Z) \]
which we denote them by two transition kernels $K_x(Z,.)$ and $K_y(Z,.)$ on $E^{\mathbb{R}^+}\times E^{\mathbb{R}^+}$. Now define a probability measure on $E^{\mathbb{R}^+}\times E^{\mathbb{R}^+}$ by the following,
\[ \tilde{\mathbb{P}} (A\times B)= \mathbb{E}( K_x(\theta_\epsilon X_x,A) K_y(\theta_\epsilon X_y,B) ) \]
This extends to a coupling of $\mathbb{P}_x$ and $\mathbb{P}_x$.

Now we have
		\[ \mathbb{P}(T>\epsilon) = \mathbb{P}(X_x(\epsilon)\ne X_y(\epsilon)) = \frac{1}{2}\|\mathbb{P}_y(\theta_\epsilon X\in .) - \mathbb{P}_x(\theta_\epsilon X\in .) \| <\epsilon \]

\end{description}
\end{proof}

\begin{proposition}
If $X_t$ is a Markov process with the local coupling property then
\begin{description}
		\item[(i)]
Every space-time harmonic function which is bounded, is continuous with respect to $x$.

		\item[(ii)]
	For any $t>0$ and any $f\in \mathcal{B}(E)$, we have $P^t f\in \mathcal{C}(E)$.

		\end{description}
\end{proposition}

\begin{proof}
\begin{description}
				\item[(i)]
Let $u(t,x)$ be a bounded harmonic function and $s>0$. Choose $t>0$ arbitrarily. We have
	\[ |u(t,y)-u(t,x)| = |\int u(s+t,z) P^t(y,dz) - \int u(s+t,z) P^t(x,dz)| \le \|u(s+t,.)\| . \|P^t(y,.)-P^t(x,.)\| \to 0\]
		\item[(ii)]
	\[ |P^t f(y)-P^t f(x)| = |\int f(z) P^t(y,dz) - \int f(z) P^t(x,dz)| \le 	\|f\| . \|P^t(y,.)-P^t(x,.)\| \to 0\]
\end{description}
\end{proof}

\section{Local Coupling Property for L\'evy Processes}
In this section we show that for L\'evy processes, the local coupling property is equivalent to absolute continuity of the transition probability measure with respect to the Lebesgue measure. We also provide a sufficient condition for local coupling property in terms of the L\'evy measure.

We first prove a lemma.

\begin{lemma}\label{lem:absolute_cont}
Let $\mu$ be a Borel measure on $\mathbb{R}$. For $a\in \mathbb{R}$ let $\mu_a$ be the translation of $\mu$ by $a$. Then $\mu$ is absolutely continuous if and only if
\[ \lim_{a\to 0} \| \mu_a - \mu \| = 0\]
\end{lemma}

\begin{proof}
Denote the Lebesgue measure on $\mathbb{R}$ by $\lambda$. To prove the if part, assume $\lambda(A)=0$. Hence for any $x\in\mathbb{R}$, $\lambda(A+x)=0$ and therefore
\begin{multline}\label{equation:proof of lemma_Fubini}
0 = \int_\mathbb{R} \lambda(A+x) d\mu(x) = \int_\mathbb{R} \int_\mathbb{R} 1_A(x+y) d\lambda(y) d\mu(x) \\
=  \int_\mathbb{R} \int_\mathbb{R} 1_A(x+y) d\mu(x) d\lambda(y) = \int_\mathbb{R} \mu_y(A) d\lambda(y)
\end{multline}
On the other hand, by assumption, the function $y\mapsto \mu_y(A)$ is continuous at $y=0$ and since is nonnegative, it follows from \eqref{equation:proof of lemma_Fubini} that $\mu(A)=\mu_0(A) = 0$.

To prove the only if part, note that if $\mu \ll  \lambda$, it follows from the Radon-Nikodym theorem that  for some $f\in L^1(\mathbb{R})$, $d\mu(x) = f(x) d\lambda(x)$ and therefore $d\mu_a(x) = f(x+a) d\lambda(x)$ which implies
\[ \| \mu_a - \mu \| = \| f(a+.)-f(.)\|_{L^1} \]
 and the right hand side tends to 0 as $a\to 0$.
\end{proof}

\begin{remark}
The special case that $\mu$ is the transition probability of a L\'evy process has been proved in~\cite{Hawkes}.
\end{remark}

\begin{theorem}\label{thm:Levy_absol}
The L\'evy process $X_t$ has the local coupling property if and only if its transition probability measure is absolutely continuous for any $t>0$.
\end{theorem}

\begin{proof}
By Theorem~\ref{thm:Markov_main} the local coupling property is equivalent to
\begin{equation}\label{eq1}
\lim_{a\to 0} \|P^t(x+a,.)-P^t(x,.)\| = 0, \forall t>0
\end{equation}
On the other hand for L\'evy processes
\[ P^t(x+a,.)= P^t(x,a+.) \]
and hence by Lemma~\ref{lem:absolute_cont}, equation \eqref{eq1} is equivalent to the absolute continuity of the transition probability measure.
\end{proof}

We state two straight forward consequences of Theorem~\ref{thm:Levy_absol}.
\begin{corollary}
	The Brownian motion has the local coupling property.
\end{corollary}
\begin{proof}
The transition probability measure is Gaussian which is absolutely continuous.
\end{proof}

\begin{corollary}\label{cor:sum_ind_Levy}
Let $X_t$ and $Y_t$ be two independent L\'evy processes and assume that $X_t$ has the local coupling property. Then so is $X_t+Y_t$.
\end{corollary}

\begin{proof}
	The transition probability of the sum is convolution of two transition probability measure. Hence if one of them is absolutely continuous, so is the sum.
\end{proof}

Now consider a general one dimensional L\'evy process with L\'evy triplet $(b,\sigma,\nu)$. It is clear that $b$ doesn't play any role in the local coupling property. It also follows from the previous two lemmas that if a $\sigma>0$ then the process has the local coupling property. Hence it remains to study the processes with triplets $(0,0,\nu)$. In what follows, we assume that $X_t$ is a L\'evy process with triplet $(0,0,\nu)$. We use the L\'evy-It\"o representation of $X_t$,
\[ X_t = \int_{|x|\le 1} x\tilde{N}(t,dx) + \int_{|x|> 1} xN(t,dx)\]
where $N$ is the Poisson random measure with intensity $dt\nu(dx)$ and $\tilde{N}$ is its compensation.

The following theorem provides a sufficient condition for local coupling property of the $X_t$ in terms of its L\'evy measure $\nu$. Recall that the minimum of two measures $\mu$ and $\nu$, denoted by $\mu\wedge\nu$ is defined as
\[ \mu\wedge\nu = \frac{1}{2} (\mu+\nu - \|\mu - \nu \|) \]
Now let $\bar{\nu}(dx)=\nu(-dx)$ and define $\rho=\nu\wedge\bar{\nu}$. Also define the auxiliary function $\eta(r)=\int_0^r x^2 \rho(dx)$.

\begin{theorem}\label{thm:Levy_suffic}
Assume that $\int_0^1 \frac{r}{\eta(r)} dr <\infty$. Then $X_t$ has the local coupling property.
\end{theorem}

\begin{remark}
	In the special case that $\nu$ is symmetric, the above theorem has been implicitly proved in~\cite{Alishahi_Salavati}.
\end{remark}

In order to prove Theorem~\ref{thm:Levy_suffic} we provide appropriate couplings between two L\'evy processes $X_t$ and $Y_t$ with characteristics $(0,0,\nu)$ and starting respectively from 0 and $a$. Without loss of generality we assume $a>0$. If we denote the distribution of $X_1$ by $\mu$, then the distribution of $Y_1$ is $\mu_a(dx)=\mu(a+dx)$.

We denote the left limit of a cadlag process $X_t$ at $t$ by $X_{t-}$ and let $\Delta X_t = X_t - X_{t-}$.

Note that it suffices to prove for the case that $\nu$ is supported in $[-1,1]$, since by the L\'evy-It\"o representation we know that $X_t$ is the sum of two independent L\'evy processes $ \int_{|x|\le 1} x\tilde{N}(t,dx)$ and $\int_{|x|> 1} xN(t,dx)$ and if the former has the local coupling property then by Corollary~\ref{cor:sum_ind_Levy} so is $X_t$. Hence from now on we assume that $\nu$ is supported in $[-1,1]$.

Let $L_1(t)$ and $L_2(t)$ be two independent L\`evy processes with characteristics  $(0,0,\nu-\frac{1}{2}\rho)$ and $(0,0,\frac{1}{2}\rho)$ and let $\tilde{N}_1(dt,du)$ and $\tilde{N}_2(dt,du)$ be the corresponding compensated Poisson random measures (cPrms). Let $\mathcal{F}_t$ be the filtration generated by $N_1$ and $N_2$ and define
\[ X_t = L_1(t) + L_2(t) \]
It is clear that $X_t$ is a L\`evy process with characteristics $(0,0,\nu)$. Now consider the following stochastic differential equation
\begin{equation}\label{definition:Y_t}
Y_t = a+ \int_0^t f(s,Y_{s-},\omega) dX_s 
\end{equation}
where $f:\mathbb{R}^+\times \mathbb{R}\times \Omega \to \mathbb{R}$ is defined by
\[f(s,y,\omega ) = \left\{ {\begin{array}{*{20}{c}}
  { - 1}&{\frac{{{X_{s - }} - y}}{2} > \,\left| {\Delta {L_2}(s)} \right|>0} \\ 
  1&{otherwise} 
\end{array}} \right.\]

Notice that by equation~\eqref{definition:Y_t}, the jumps of $X$ and $Y$ occur at the same times and have the same magnitude but probably different directions.

Note that the classical existence theorems for solutions of SDEs are not applicable to equation~\eqref{definition:Y_t} since $f$ is not a Lipschitz (not even continuous) function of $y$. In order to prove the existence of solution, we rewrite~\eqref{definition:Y_t} as an SDE with respect to cPrms as follows,

\begin{equation*}
Y_t = a + \int_0^t \int_\mathbb{R} u \tilde{N}_1(ds,du) \\
+ \int_0^t \int_\mathbb{R}  g(s,Y_{s-},u,\omega) u \tilde{N}_2(ds,du)
\end{equation*}
where
\[ g(s,y,u,\omega) = \chi_{\frac{|X_{s-}-y|}{2}<|u|} - \chi_{\frac{|X_{s-}-y|}{2}\ge |u|}  \]
Now we can use the existence result for such equations (see e.g. Applebaum~\cite{Applebaum}, Theorem 6.2.3). For that, we need to prove linear growth and Lipschitz conditions for coefficients. To prove this, note that both coefficients are almost surely bounded by $|u|$. Hence the Lipschitz condition reduces to $\int_\mathbb{R} |u|^2 \nu (du) < \infty$ which holds by L\'evy condition and the assumption that $\nu$ is supported in $[-1,1]$. The linear growth condition is similar. Hence the equation has a square integrable and cadlag solution.

Now we claim that
\begin{lemma}
	$Y_t$ is a L\`evy processes with characteristics $(0,0,\nu)$ and starting from $a$.
\end{lemma}
\begin{proof}
The idea is that $Y_t$ and $X_t$ has the same movements except that $Y$ sometimes jumps in the opposite direction than $X$, but notice that if we decompose $\nu$ as $(\nu-\rho)+\rho$, the jumps of $X$ that come from the $\rho$ part have a symmetric distribution and at exactly these jumps $Y$ makes an opposite jump with probability $\frac{1}{2}$.

In order to give a rigorous proof, we calculate the conditional characteristic function of $Y_t$, i.e. $\mathbb{E}(e^{i\xi Y_t}|\mathcal{F}_s)$.

We have
\[ Y_t = a + \int_0^t \int_\mathbb{R} u \tilde{N}_0(ds,du) \\
+\int_0^t \int_\mathbb{R} u \tilde{N}_3(ds,du)+ \int_0^t \int_\mathbb{R}  g(s,Y_{s-},u,\omega) u \tilde{N}_2(ds,du) \]

We write the It\^o's formula for $e^{i\xi Y_t}$,
\begin{multline*}
e^{i\xi Y_t} - e^{i\xi Y_s} = \int_s^t \int_\mathbb{R} \left( e^{i\xi Y_{r-} + i\xi u} - e^{i \xi Y_{r-}}\right) \tilde{N}_1(dr , du)\\
+ \int_s^t \int_\mathbb{R} \left( e^{i\xi Y_{r-} + i\xi g(r,Y_{r-},u) u} - e^{i \xi Y_{r-}}\right) \tilde{N}_2(dr , du)\\
+ \int_s^t \int_\mathbb{R}  e^{i\xi Y_{r-}} [e^{i\xi u}-1- i\xi u] (\nu-\frac{1}{2}\rho)(du)dr\\
+ \int_s^t \int_\mathbb{R} [e^{i \xi g(r,Y_{r-},u) u} - 1 - i\xi g(r,Y_{r-},u) u]\frac{1}{2}\rho(du)dr
\end{multline*}
taking expectations conditioned on $\mathcal{F}_s$ and noting that the first two integrals are martingales we find,
\begin{multline} \label{eq:proof_characteristic_1}
\mathbb{E}(e^{i\xi Y_t}|\mathcal{F}_s) - e^{i\xi Y_s} =
\int_s^t  \mathbb{E} \bigg(  e^{i\xi Y_{r-}} \int_\mathbb{R} \Big( [e^{i\xi u}-1- i\xi u](\nu-\frac{1}{2}\rho)(du)\\
+ [e^{i \xi g(r,Y_{r-},u) u} - 1 - i\xi g(r,Y_{r-},u) u] \frac{1}{2}\rho(du) \Big)  |\mathcal{F}_s \bigg)  dr
\end{multline}
Now note that $\rho$ is a symmetric measure and for each $\omega$, the function $u\mapsto g(r,Y_{r-},u,\omega)$ is an even function with values $\pm 1$, hence we have
\[ \int_\mathbb{R}  \left( e^{i \xi g(r,Y_{r-},u) u} - 1 - i\xi g(r,Y_{r-},u) u \right) \rho(du) = \int_\mathbb{R}  \left( e^{ i\xi u}-1 -i\xi u \right) \rho(du) \]

Substituting in~\eqref{eq:proof_characteristic_1} we find,
\begin{multline*}
\mathbb{E}(e^{i\xi Y_t}|\mathcal{F}_s) - e^{i\xi Y_s} = 
\int_s^t  \mathbb{E} \bigg(  e^{i\xi Y_{r-}} \int_\mathbb{R}  [e^{i\xi u}-1- i\xi u] \nu(du) |\mathcal{F}_s \bigg) dr \\
=\psi(\xi) \int_s^t \mathbb{E} \left( e^{i\xi Y_{r-}} |\mathcal{F}_s \right)dr
\end{multline*}

where $\psi(\xi)=\int_\mathbb{R}  [e^{i\xi u}-1- i\xi u] \nu(du)$. Note that $\psi(\xi)$ is indeed the characteristic exponent of $X$.

Hence if we define $h(t)=\mathbb{E}(e^{i\xi Y_t}|\mathcal{F}_s)$, $h$ satisfies the ordinary differential equation $h^\prime(t)=\psi(\xi)h(t)$. This ode has the unique solution
\[ \mathbb{E}(e^{i\xi Y_t}|\mathcal{F}_s) = e^{i\xi Y_s} e^{(t-s)\psi(\xi)} .\]
The last equality implies easily that $Y$ is a L\'evy process with characteristics $(0,0,\nu)$.
\end{proof}

Now let
\[ Z_t = Y_t - X_t \]
By subtracting the integral representations of $X_t$ and $Y_t$, we find that $Z_t$ satisfies the following sde,
\[ Z_t = -2 \int_0^t \int_\mathbb{R} \chi_{\frac{|Z_{s-}|}{2}\ge |u|} u \tilde{N}_2(ds,du) \]
It is clear from the above equation that the jumps of $Z_t$ always have magnitudes less than $|Z_t|$. Hence since $Z_0=a>0$, $Z_t$ is always non-negative.

Now we define two stopping times
\[ \tau_a = \inf \{ t: Z_t = 0 \}, \]
\[ \bar{\tau}_a = \inf \{ t: Z_t \notin (0,1) \}. \]
Since the jumps of $Z_t$ satisfy $| \Delta Z_s| \le Z_{s^-}$ it is clear that $Z_{\bar{\tau}_a} \le 2$.
Note also that since $\rho$ is symmetric, the jumps of $Z_t$ have a symmetric distribution.
\begin{lemma}\label{lemma:limit}
\[ \lim_{t\to\infty} Z_t =0,\quad a.s. \]
\end{lemma}
\begin{proof}
Since $X_t$ and $Y_t$ are martingales hence so is $Z_t$. Moreover, $Z_t$ is non-negative and hence by martingale convergence theorem it has a limit $Z_\infty$, as $t\to\infty$. On the other hand, by the assumption made on $\eta$, for any $\epsilon>0$ we have $\nu((0,\epsilon))>0$. Hence jumps of size greater than $\epsilon$ occur in $X_t$ with a positive rate. This implies that if $Z_\infty=\alpha \ne 0$, then $Z_t$ has infinitely many jumps greater than some $\epsilon$ with $0<\epsilon<\frac{\alpha}{2}$ which contradicts its convergence. Hence $Z_\infty = 0\quad a.s$.
\end{proof}

We define an auxiliary function $g:[0,\infty)\to\mathbb{R}$ by letting
\[ g(x) = \int_x^1 \int_y^1 \frac{1}{\eta(r)} dr dy \]

\begin{lemma}\label{lemma:g}
If $\int_0^1 \frac{r}{\eta(r)} dr <\infty$ then $g$ is defined on $[0,\infty)$. Moreover, it is differentiable on $(0,\infty)$ and its derivative is absolutely continuous and $g^{\prime\prime} (x)= \frac{1}{\eta(x)}$ for almost every $x$. Furthermore, for every $x,y\in[0,\infty)$, we have
\[ g(y)-g(x) \ge g^\prime(x) (y-x) + \frac{1}{2} \frac{1}{\eta(x)} (y-x)^2 1_{y<x} \]
\end{lemma}

\begin{proof}
By Fubini's theorem,
\[ g(x) = \int_x^1 \int_x^r \frac{1}{\eta(r)} dy dr = \int_x^1 \frac{r}{\eta(r)} dr \]
hence $g$ is finite and differentiable everywhere and its derivative is absolutely continuous and $g^{\prime\prime} = \frac{1}{\eta} \quad a.e $.

To prove the last claim, note that since $g^\prime=-\int_x^1 \frac{1}{\eta}$ is increasing hence $g$ is convex and therefore,
\[ g(y)-g(x) \ge g^\prime(x) (y-x) \]
If $y<x$, by the integral form of the Taylor's remainder theorem we have
\[ g(y)-g(x) = g^\prime(x) (y-x) + \int_y^x (t-y) g^{\prime\prime}(t) dt \]
where since $g^{\prime\prime}(t)$ is decreasing we have $g^{\prime\prime}(t) \ge \frac{1}{\eta(x)}$ which implies that in the case that $y<x$,
\[ g(y)-g(x) \ge g^\prime(x) (y-x) + \frac{1}{2} \frac{1}{\eta(x)} (y-x)^2\]
and the proof is complete.
\end{proof}

\begin{lemma}\label{lemma:tau_bar}
If $\int_0^1 \frac{r}{\eta(r)} dr <\infty$ then
\[ \lim_{a\to 0} \mathbb{E} \bar{\tau}_a = 0\]
\end{lemma}
\begin{proof}
We have by Lemma~\ref{lemma:g},
\begin{multline}
	g(Z_{t \wedge \bar{\tau}_a}) = g(a) + \sum_{s\le t\wedge \bar{\tau}_a} \Delta g(Z_s)\\
	\ge g(a) + \sum_{s\le t\wedge \bar{\tau}_a} g^\prime(Z_{s-})\Delta Z_s + \frac{1}{2} \sum_{s\le t\wedge \bar{\tau}_a} \frac{1}{\eta(Z_{s-})}  (\Delta Z_s)^2 1_{\Delta Z_s<0}
\end{multline}

Since the second term on the right hand side is a martingale, we have
\begin{equation} \label{equation:Ito's formula2}
\mathbb{E} g(Z_{t \wedge \bar{\tau}_a}) \ge g(a) +  \frac{1}{2} \mathbb{E} \sum_{s\le t\wedge \bar{\tau}_a} \frac{1}{\eta(Z_{s-})}  (\Delta Z_s)^2 1_{\Delta Z_s<0}.
\end{equation}

Noting that the jumps of $Z_s$ are independent and have symmetric distribution, we conclude
\[  \mathbb{E} \sum_{s\le t\wedge \bar{\tau}_a} \frac{1}{\eta(Z_{s-})}  (\Delta Z_s)^2 1_{\Delta Z_s<0} = \frac{1}{2} \mathbb{E} \int_0^{t\wedge \bar{\tau}_a} \frac{1}{\eta(Z_{s-})} \eta(Z_s) ds =  \frac{1}{2} \mathbb{E} (t \wedge \bar{\tau}_a)\]

Substituting in~\eqref{equation:Ito's formula2} implies

\[	\frac{1}{4} \mathbb{E} (t \wedge \bar{\tau}_a) \le \mathbb{E}g(Z_{t \wedge \bar{\tau}_a}) - g(a) \]

Now letting $t\to \infty$ and noting that $Z_s$ is uniformly bounded by 2 for $t\le \bar{\tau}_a$, we find that
\[	\frac{1}{4} \mathbb{E} (\bar{\tau}_a) \le \mathbb{E}g(Z_{\bar{\tau}_a}) - g(a) \]
Now let $a\to 0$. By continuity of $g$, we have $g(a)\to g(0)$. On the other hand, \[ \mathbb{P}(Z_{\bar{\tau}_a}\ge 1) \le \mathbb{E}(Z_{\bar{\tau}_a}) = a\to 0 \]
where we have used the optional stopping theorem. Now we can write,
\begin{multline}
\mathbb{E}g(Z_{\bar{\tau}_a}) = \mathbb{E}g(Z_{\bar{\tau}_a}; Z_{\bar{\tau}_a}=0) + \mathbb{E}g(Z_{\bar{\tau}_a};Z_{\bar{\tau}_a}\ge 1) \\
\le g(0) + 2 \mathbb{P}(Z_{\bar{\tau}_a}\ge 1) \to g(0)
\end{multline}
Hence we find that,
\[ \lim_{a\to 0} \mathbb{E} \bar{\tau}_a = 0\]
\end{proof}

We are now ready to prove Theorem~\ref{thm:Levy_suffic}.
\begin{proof}[Proof of Theorem~\ref{thm:Levy_suffic}]
It suffices to prove that for any $\epsilon>0$,
\[\lim_{a\to 0} \mathbb{P} (\tau_a\ge \epsilon) = 0.\]
We have
\[ \mathbb{P} (\tau_a\ge \epsilon) \le \mathbb{P} (\bar{\tau}_a\ge \epsilon) + \mathbb{P} (\bar{\tau}_a \le \epsilon, Z_{\bar{\tau}_a}\ge 1) \]
The first term on the right hand side is less than or equal to $\mathbb{E} \bar{\tau}_a/\epsilon$ and the second term is less than or equal
\[ \mathbb{P} (\sup_{0\le t\le \epsilon} Z_t \ge 1)\]
which by Doob's maximal inequality is less than or equal to $\mathbb{E} Z_\epsilon = a$. Hence,
\[ \mathbb{P} (\tau_a\ge \epsilon) \le \mathbb{E} \bar{\tau}_a/\epsilon + a \]
and from Lemma~\ref{lemma:tau_bar} the statement follows.

\end{proof}

\end{document}